\documentclass[a4paper, 11pt]{article}
\usepackage{calc, amsmath, amsthm, a4, latexsym, amssymb,  url, tabularx} %MnSymbol, 
\usepackage[dvips]{graphicx}
\usepackage{soul}
\usepackage{color}
\newtheorem{thm}{Theorem}[section]
\newtheorem{lemma}[thm]{Lemma}

\newtheorem{prop}[thm]{Proposition}

\theoremstyle{definition}

\newtheorem{defi}[thm]{Definition}

\newtheorem{rem}[thm]{Remark}

\newcommand{\be}[1]{\begin{equation}\label{#1}}
\newcommand{\ee}{\end{equation}}
\newcommand{\ba}{\begin{array}}
\newcommand{\ea}{\end{array}}
\newcommand{\bal}{\begin{aligned}}
\newcommand{\eal}{\end{aligned}}

\newcommand{\N}{\mathbb{N}}

\newcommand{\Z}{\mathbb{Z}}
\newcommand{\E}{\mathbb{E}}

\newcommand{\p}{\mathbb{P}}
\renewcommand{\P}{\mathbb{P}}

\newcommand{\cA}{\mathcal{A}}

\newcommand{\cD}{\mathcal{D}}
\newcommand{\cE}{\mathcal{E}}

\newcommand{\cG}{\mathcal{G}}
\newcommand{\cH}{\mathcal{H}}

\newcommand{\cS}{\mathcal{S}}
\newcommand{\cT}{\mathcal{T}}

\newcommand{\cF}{\mathcal{F}}

\newcommand{\ft}{\mathfrak{t}}

\newcommand{\1}{1\hspace{-0.098cm}\mathrm{l}}
\newcommand{\ind}{1\hspace{-0.098cm}\mathrm{l}}
\newcommand{\la}{\lambda}
\newcommand{\Om}{\Omega}
\newcommand{\eps}{\varepsilon}
\newcommand{\om}{\omega}
\renewcommand{\phi}{\varphi}

\newcommand{\ssup}[1] {{\scriptscriptstyle{({#1}})}}

\newcommand{\ra}{\rightarrow}
\newcommand{\gw}{T}%{{\rm GW}}

\newcommand{\hopcount}{graph }

\newcommand{\wei}{\mathrm{w}}

\newcommand{\sumtwo}[2]{\sum_{\substack{#1 \\ #2}}} % sum with 2 lines
\begin{document}

\begin{center}
{\Large \bf
Local neighbourhoods for first passage percolation\\[2mm]on the 
configuration model
}\\[5mm]

\vspace{0.7cm}
\textsc{Steffen Dereich\footnote{Institut f\"ur Mathematische Statistik,
Westf\"alische Wilhelms-Universit\"at M\"unster, Einsteinstra\ss{}e 62, 48149 M\"unster, Germany.
}} 
and \textsc{Marcel Ortgiese\footnote{Department of Mathematical Sciences, University of Bath, Claverton Down, Bath, BA2 7AY,
United Kingdom.}
} 
\\[0.8cm]
{\small(Version of November 17, 2017)} 
\end{center}

\vspace{0.3cm}

\begin{abstract}
  \noindent 
We consider first passage percolation on the configuration model. Once the network has been generated each edge is assigned an i.i.d.\ weight modeling the passage time of a message along this edge. Then independently two vertices are chosen uniformly at random, a sender and a recipient, and all edges along  the  geodesic connecting the two vertices are coloured in red (in the case that both vertices are in the same component). 
In this article we prove local limit theorems for the coloured graph around the recipient in the spirit of Benjamini and Schramm. We consider the explosive regime, in which case the random distances are of finite order, and the Malthusian regime, in which case the random distances are of logarithmic order.

\end{abstract}

\section{Introduction and main results}

\subsection{Introduction}

Originally, first passage percolation  was introduced in 1965 by Hammersley and Welsh~\cite{HamWel65}  as  a  model  for a  fluid  flow  through  a  random  medium.  
Generally, for a given graph (in the original example the lattice $\Z^d$)  one assigns each edge a strictly positive i.i.d.\ weight and then endows the graph with the metric induced by the weights. In the original model the weights represent the time it takes for the fluid to flow though the edge and a significant amount of research was concerned with the structure and length of geodesics for vertices that are far apart. We refer to \cite{auffinger2017} for a recent review.

In this article we focus on the case where the graph (network) itself is generated at random. More explicitly, we consider the configuration model where first  finitely many half-edges are attached to the vertices of a finite vertex set and then all half-edges are paired uniformly at random. The model has attracted significant attention recently, since on one hand it is possible to generate graphs with heavy tailed degree distributions (a  phenomenon  observed in real world networks) by choosing the half-edges  appropriately and  since on the other hand the uniform pairing of the half-edges has nice stochastic properties making the analysis feasible, see   \cite{RemcoNotes2, RemcoStFlour}. In this context, the weights  may describe 
the time it takes for a disease or rumour to spread or the cost for the transmission of a message along an edge. 

In the analysis  of first passage percolation on complex networks such as the configuration model one chooses on a large graph independently two vertices uniformly at random, say a recipient and sender,  and asks for properties of the minimal weight path, the geodesic, connecting the two vertices. In recent years, significant  progress has been made  for a variety of random graph models, see e.g.~\cite{RemcoNotes2, RemcoStFlour} for surveys.
The distance of minimal weight paths and the corresponding number of hops  
are analysed in the 
Erd\"os-Renyi graph first in the dense~\cite{Hofstad2001, Bhamidi2008}, but also in the sparse case~\cite{bhamidi2011}. For the configuration model these questions were first answered for the easier case of exponential weights~\cite{BHH10, BHH_extreme_10}, where the memory-less property enables an approximation of the local structure by a Markovian branching process. It turns out that the behaviour observed here is universal. For a wide class of distribution of weights, the same scaling  of distances leads to convergence, where however the limiting law depends on the weight distribution, see~\cite{BHH12}.
An extension of these techniques allows to answer more complicated questions about the geometry induced by the shortest path: In~\cite{BGHK15} the authors consider  the shortest-path tree obtained when starting from a single vertex and always following shortest edge weights. Then, the authors give an explicit description of the law of the degree of a uniformly chosen vertex in the shortest-path tree.

Assuming that the distribution generating the weights has no atoms, there is, almost surely, a unique minimal weight  path (geodesic) connecting recipient and sender provided  that both are in the same component. Our interest lies in the stochastic interplay between the local neighbourhood around the recipient and the geodesic connecting recipient and sender.
Our research is intimately related to the following statistical question. Given the local neighbourhood around the recipient, how likely is it that a rumour will spread to the recipient along a particular path in the local neighbourhood?
More formally, we encode the information of the geodesic by colouring all edges passed by the geodesic in red and we derive  local limit theorems (in the spirit of Benjamini-Schramm, see~\cite{BS01} and~\cite{AldousSteele}) for the coloured, in the recipient rooted random graphs.

\subsection{Local convergence of coloured graphs}\label{ssn:local_conv}

In this section we review some of the basic graph theoretic notation and 
construct the coloured graph that describes the local geometry of geodesics
induced by first passage percolation. Finally, we define a notion of local convergence
adapted to coloured graphs.

 Let us fix the notation.   In the following the identifiers $G,G_1,\dots$ refer to locally finite, nonempty random or deterministic, weighted multigraphs, briefly called \emph{standard graphs}. The corresponding sets of vertices will be denoted by $V,V_1,\dots$, the sets of edges by $E,E_1,\dots$ and the  weights  by $\wei,\wei_1,\dots$.   There are two natural notions of distance on a standard graph $G$:  the first one is the \hopcount distance, which counts the number of edges on the shortest path between two vertices. However,  we will be mostly interested in the distance
 $d_G$  induced by the weights $\wei$, which is defined for two vertices $v_1,v_2\in V$ by setting
$$
d_G(v_1,v_2)=\inf\Bigl\{\sum_{k=1}^n  \wei(e_k): n\in \N \text{ and } (e_1,\dots,e_n)\text{ path joining }v_1\text{ and }v_2\Bigr\}.
$$

A sequence $(V_1,E_1), (V_2,E_2),\dots$ of finite random multigraphs will be called \emph{network model} and when speaking of \emph{first-passage percolation} we will always assume that $G_1,G_2,\dots$ refer to the standard graphs obtained from $(V_1,E_1),(V_2,E_2),\dots$ by independently assigning  i.i.d.\ weights to the individual edges. The distribution generating the weights will not depend on the graph and  will be denoted by $\mu$.  We will assume that $\mu$ has no atoms so that, almost surely, each pair of vertices has at most one geodesic connecting it. We stress that for all our statements probabilistic  dependencies between the individual graphs $G_1,G_2,\dots$ are irrelevant.

A standard tool in the analysis of network models is the derivation of local limit theorems in the sense of Benjamini-Schramm \cite{BS01}. The concept plays an analogous role to Palm measures for stationary point processes. It 
 describes a weak limit theorem for the graph  centered at an independent uniformly chosen vertex. We extend this concept by marking the geodesic connecting the latter vertex to another independent uniformly chosen vertex in red.

\begin{defi}Let $G$ be a  finite standard graph. 
\begin{enumerate}\item We call a random rooted graph  $(G,o)$ with $o$ being a random vertex such that given $V$, $o$ is uniformly distributed on $V$, a \emph{neighbourhood for $G$}. 
 \item Let  $(G,o)$ be a \emph{neighbourhood for $G$} and $u$ be a random vertex such that given $(G,o)$, $u$ is uniformly distributed on $V$. Then the random rooted coloured graph $\cG=(G,o,c)$ with  $c$ being a random mapping (colouring) $c:E\to\{0,1\}$ such that for an edge $e\in E$
 $$
 c(e)=1 \ \Leftrightarrow \ e\text{ lies on a $d_G$-geodesic connecting $o$ and $u$}
 $$
 is called a  \emph{geodesic neighbourhood for $G$}. 
\item For $R\in\N_0$ and a random rooted and  coloured  standard graph $\cG=(G,o,c)$ we call the rooted and coloured  standard graph $\cG|_{\leq R}$  obtained from $\cG$ by removing all vertices with \hopcount distance strictly bigger than $R$ to $o$ together with the attached edges the \emph{$R$-truncation of $\cG$}.
 Analogously, we define the $R$-truncation $(V,E,o)|_{\leq R}$ of any rooted graph $(V,E,o)$ and denote by 
$V|_{\leq R}$ and $E|_{\leq R}$ the set of vertices and edges of the $R$-truncation.
\end{enumerate}\end{defi}

Note that formally a coloured local neighbourhood $(G,o,c)$ for $G$ is characterized by the conditional distribution of $(o,c)$ given $G$.
Generally, we refer to the edges $e\in E$ with $c(e)=1$ as \emph{red edges} and to the remaining ones as \emph{black edges}. Although $u$ is not part of the definition of a coloured graph we will always assume it to be defined as above in our considerations.

 We introduce a topology on the space of random rooted and coloured standard graphs.

\begin{defi}  A sequence of random rooted and coloured standard graphs  $(\cG_n)_{n\in\N}$ converges \emph{locally} to a random rooted and  coloured standard graph $\cG$, if for every $R\in\N$ as $n\to\infty$
\[ d_{\rm TV} (\cG_n|_{\leq R} , \cG|_{\leq R}) := 
\inf  \p ( \cG_n|_{\leq R} \not\sim \cG|_{\leq R} ) \ra 0 , \]
where the infimum is taken over all couplings of $\cG_n|_{\leq R}$ and $\cG|_{\leq R}$, and $\cG_n|_{\leq R} \sim \cG|_{\leq R}$ means that there is an isomorphism that preserves 
the underlying graph structure including the roots, the colours and the weights. 
\end{defi}

\begin{rem}
In the previous definition we use the concept of convergence in total variation distance rather than weak convergence. Therefore the topology restricted to the weighted graphs is stronger than its analogue introduced in~\cite{AldousSteele}. However the topology restricted to the multigraph (without weights and colouring)  is just classical Benjamini-Schramm convergence.
\end{rem}

\subsection{Main results}\label{ssn:main_results}

We consider the configuration model. 
Let $V$ be a finite set and $d:V\to \N_0$ a mapping  such that
\[ \ell := \sum_{i=1}^n d_i ,\]
is even. We generate a random multigraph graph $(V,E)$ by 
\begin{itemize} \item taking a random uniform pairing $\mathbb H$ of the set 
$$
\{(v,j):v\in V, j=1,\dots, d_j\}
$$
and 
\item  interpreting each unordered pair $\langle (v,j), (v',j')\rangle $ of $\mathbb H$  as an  undirected edge~$\langle v,v'\rangle$.
\end{itemize}
A random graph with the corresponding distribution is called \emph{$(V,d)$-configuration graph}. It is also possible to first choose the parameters $V$ and $d$ at random and then generate the random graph according to the above rule. In that case we call the graph a random configuration model. The resulting multigraph has very few self-loops
and multiple edges, 
see~\cite[Chapter 7]{RemcoNotes1} for more details.

\begin{defi}Let $\cD$ be an integrable distribution on $\N_0$.
A sequence of random configuration models $(V_n,E_n)_{n\in\N}$ with $\#V_n=n$ such that
$$\cD_n=\frac1n\sum_{i\in V_n} \delta_{d_i^{(n)}}\Rightarrow \cD\text{, \  in probability,}$$
and the mean of $\cD_n$ converges in probability to the mean of $\cD$,
is called \emph{configuration network with asymptotic degree distribution $\cD$}. If for a $\cD$-distributed random variable~$D$
\begin{equation}\label{eq:mean}
\E[D(D-1)]/\E[D]>1,
\end{equation}
then we call the network \emph{supercritical}. 
Further, if  the family $(D_n^2 \log D_n)_{n\in\N}$ is uniformly integrable, where $D_n$ denotes the degree of a uniformly chosen vertex of $G_n$, then we call the configuration model $(G_n:n\in\N)$ \emph{regular}.
\end{defi}

A supercritical configuration network  has a giant component in the sense  that if
$C_{\rm max}^\ssup{n}$ denotes the largest component in the configuration model $(V_n, E_n)$, then 
\[ |C_{\rm max}^\ssup{n}|/n \ra \zeta  , \quad\mbox{ in probability},\]
where $\zeta \in (0,1]$ is the survival probability of a suitable branching process that we describe next, see~\cite[Thm.~4.1]{RemcoNotes2}.

For first-passage percolation on a configuration network with asymptotic degree distribution $\cD$ it is well known that the local  neighbourhood   converges locally to a (modified) Galton-Watson process $\gw$  with independent $\mu$-distributed weights.
To introduce the limiting process $\gw$ we denote by $D$ a $\cD$-distributed random variable and by $D^*$ its size-biased counterpart, i.e.\
\[ \p (D^* = k ) = \frac{k\, \p ( D =  k ) }{\E [D]} ,\quad k \in \N_0 .  \]
We exclude the degenerate case where $D=0$, almost surely.
The limit $\gw$ is a rooted weighted  tree $T=(V,E,\wei, o)$ that can be generated as follows.  The root $o$  has a random number of children with the same distribution as $D$. 
Each of its descendants has an independent number of offspring with the same distribution as $D^*-1$.
We interpret the branching process as a graph by drawing an edge between each individual and its offspring and we assign to each edge $e$ an independent $\mu$-distributed weight~$\wei_e$.

So far we described the classical Benjamini-Schramm limit (including the weights).
It remains to introduce the colouring.
We distinguish two regimes:

\begin{tabularx}{\textwidth}{ p{1.2cm} >{\raggedright\arraybackslash}X}%
{\bf (EXP) } & \emph{Explosion}: The branching process $T$ explodes in finite time with strictly positive probability, meaning that the probability that the explosion time
 \[ \tau:=\inf \{t\geq 0: \#\{v: d_w(o,v)<t\} =\infty\}  \]
 is finite is strictly positive.\\
{\bf (MG) } & \emph{Malthusian growth}: The configuration network is 
 supercritical and regular. 
\end{tabularx}

{\bf Limiting object in the case with explosion}.

First-passage percolation on a Galton-Watson process is said to explode, if the explosion time
 \[ \tau:=\inf \{t\geq 0: \#\{v: d_w(o,v)<t\} =\infty\}  \]
is finite with strictly positive probability.
For  some of the properties of branching processes see~\cite{komjathy_explosion} and 
for an characterization of offspring and weight distributions that lead to explosion see~\cite{Amini13}.

In the case of explosion one has 
$$
\P(T \text{ is finite or explodes})=1.
$$
Further, conditionally on the event $\{T \text{ explodes}\}$, there exists a unique random infinite geodesic ray 
$\xi = (\xi_0 = o, \xi_1, \xi_2, \ldots )$ (path to explosion), i.e.\ a semi-infinite geodesic  path in the tree, such that
$$
\tau=
\lim_{n\to\infty} d_T(o,\xi_n).
$$

We colour the graph $T$ as follows. We independently toss a coin with success probability given by the 
survival probability of $\gw$. If successful and if the tree $\gw$ is infinite,  we colour all edges $(\xi_{j-1},\xi_j)$ on the path to explosion in red and the remaining edges in black.
If the coin toss is unsuccessful or $T$ is finite, 
we colour all edges black.
In all cases, we denote the induced 
rooted coloured tree by $\cT$.
The case that all edges are black corresponds to the case  where the source or receiving vertex
is not in the giant component and there is no geodesic in the graph.

{\bf Limiting object in the case with Malthusian growth.}

We now assume that the configuration model is supercritical and regular.
By supercriticality, one has $\E[D^\star-1]>1$ and there exists a unique $\lambda>0$ that solves the equation
$$
\E[D^\star-1] \int_{(0,\infty)} e^{-\lambda x} \, \mu(dx )=1,
$$
the \emph{Malthusian parameter}.
In this case we can equip each vertex $v$ of the tree $T$
with the martingale limit
\begin{equation}\label{eq:mg_limit}
M(v):= \lim_{n\to \infty} 
\sumtwo{v< v'}{|v'| = n}
e^{-\lambda d_w(v,v')}, 
\end{equation}
where, we denote by $|u|$ the \hopcount distance to the root 
and we write $u < v$ if $u$ is an ancestor of $v$ (i.e.\ $u$ is on the shortest path in the \hopcount distance from $v$ to $o$).

Then, we have the consistency relation that for every vertex $v$ and $n \in \N$ with $|v|\leq n$
$$
M(v)=\sumtwo{|v'| = n}{v\leq v'} e^{-\la d_w(v,v') } M(v') .
$$
For later reference we denote by $M^*$ a random variable that is identically distributed as $M(v)$
for $|v| =1$ (conditionally on $\{ v \in V \, : \, | v|  =1 \} \neq \emptyset$).
Due to  the regularity assumption $\E[D^2 \log_+ D] < \infty$ (equivalently $\E [ D^\star \log_+D^\star] < \infty$) by~\cite{JagersNerman1984} 
 one has,  up to nullsets
$$
\#\gw=\infty \ \Leftrightarrow \ M(o)>0 . 
$$
Given the tree $\gw$, 
we first independently toss a coin with success 
probability given by the survival probability of $\gw$. 
If successful and if the tree is infinite, then
we colour a unique ray starting from the root in red with conditional distribution 
\begin{align}\label{eq734}
\P(v\text{ on red ray} \, |\, T)= \frac {e^{-\lambda d_w(\rho,v)}M(v)}{M(\rho)}= \frac {e^{-\lambda d_w(\rho,v)}M(v)}{\sum_{v':|v'|=|v|}e^{-\lambda d_w(\rho,v')} M(v')}
\end{align}
for every vertex $v$ of the tree $\gw$.
If the coin toss is unsuccessful or if
 the tree is finite, we colour all  edges  black. 
We denote by $\cT$ the corresponding rooted coloured tree. 

As before, the case that all edges are black corresponds to the case where the source or receiving vertex
are not in the giant component and there is no geodesic in the graph.

Now, we are finally ready to state our main theorem. 

\begin{thm}\label{thm:local_conv} Let $(G_n:n\in\N)$ be first-passage percolation on the configuration network with 
asymptotic degree distribution $\cD$.  
 If either the branching process $T$  explodes or the configuration network is  supercritical and regular, then the corresponding 
 geodesic neighbourhoods
 $(\cG_n:n\in\N)$ converge 
locally to the rooted coloured  tree  $\cT$  introduced  above.
\end{thm}

The remaining paper is structured as follows. 
In Section~\ref{sec:proof_no_explosion}, we prove Theorem~\ref{thm:local_conv}
in the case without explosion.
In Section~\ref{sec:explosion}, we prove Theorem~\ref{thm:local_conv}
in the case with explosion.

\section{The case without explosion}\label{sec:proof_no_explosion}

We use the following proposition of  \cite[Proposition 3.2]{BGHK15} (considerably relying on \cite{BHH12}), which describes the distances between a uniformly chosen source vertex $V$ and 
a finite number of other (target) vertices.

\begin{prop}\label{prop:malthusian_dist} Let $(G_n:n\in\N)$  be first-passage percolation on a regular supercritical  configuration model with asymptotic degree distribution $\cD$. 
Let  $v$ be a uniformly chosen vertex and further let $v_1,\dots,v_k$ denote distinct vertices for which the degrees $(d_{v_1},\dots,d_{v_k})$ converge jointly in distribution to independent copies of $D^\star-1$ and that are independent of the pairing of the half-edges. Then there exists a random sequence  $(\lambda_n)$ converging to $\la$ in probability such that
$$
(\lambda_n d_{G_n}(v,v_i) - \log n)_{i=1,\dots,k} \Rightarrow    (\log  \cE_i/W_i+ \log 1/\hat W+c)_{i=1,\dots,k},
$$
where $\cE_1,\dots,\cE_k,W_1,\dots,W_k,\hat W$ are independent random variables  with
$$
\cE_i\sim \mathrm{Exp}(1), \ W_i \stackrel d= M^*,\text{ and } \hat W \stackrel d= M(o), 
$$
 and $c$ is an explicit constant depending on $\mu$ and $\cD$.
Moreover, the probability that $\hat W = 0$ (resp.\ $W_i = 0$) corresponds to the limiting probability of $V$ (resp.\ $V_i$) not being in the largest component.
\end{prop}

\begin{proof}[Proof of Theorem~\ref{thm:local_conv} in the regular case]
Without loss of generality we assume that the set of vertices of $G_n$ is equal to $\{1,\dots,n\}$ and that the distribution of $G_n$  is invariant under every permutation of  the labels (e.g.\ by randomly permuting the labels). In particular, this guarantees that we cannot infer information about the degree of individual vertices by knowing their label.

Let $o_n$ and $u_n$ be  independent uniformly chosen vertices from $G_n$ and  
fix $R\in\N$. 
 We write $G^*_n$ for the graph obtained from $G_n$ when removing all edges of  $E_n|_{\leq R}$.

By the standard coupling construction, see e.g.~\cite[Chapter 5.2.1]{RemcoNotes2}, there exists a coupling of $(G_n,o_n)$ and $\gw$  such that for the coupled random variables and   $\Om_n=\{(G_n,o_n)|_{\leq R}\sim T|_{\leq R}\}$, $\lim_{n\to\infty}\P(\Om_n)= 1$ and on $\Om_n$ given $(G_n,o_n)|_{\leq R}$ and $T|_{\leq R}$
\begin{itemize}
	\item $G_n^*$ is again a configuration model and
	\item 
	the rooted and weighted tree $\gw$ is generated by growing 
	 from each vertex $v$ with $|v|=R$  an independent Galton-Watson process  $T(v)$ with offspring distribution $D^\star -1$ and attaching independent weights to all edges.
\end{itemize}

Let $\mathfrak t$ be a finite weighted, rooted tree (coded using the Ulam-Harris notation) with depth at most $R$. Denote by $v_1, \ldots, v_K$ the vertices in $\ft$ with \hopcount distance $R$ to the root and set $v_0 = o$.
We continue under the regular probability distribution $\P_n^\mathfrak{t}=\P(\,\cdot\, | \,G_n|_{\leq R}\sim T=\mathfrak t)$. We denote by $\phi_n$ a random $G_n|_{\leq R}$-measurable isomorphism between the tree $\mathfrak t$ and $G_n|_{\leq R}$ and set $v_k^\ssup{n}=\phi_n(v_k)$ for $k=0,1,\dots,K$.
Under~$\P_n^\mathfrak {t}$,  $(G_{n}^*:n\in\N)$ is a configuration network with asymptotic degree distribution~$\cD$. 
As is well known (see for instance~\cite[Chapter 5.2]{RemcoNotes2}) the degrees of the vertices $v_1^\ssup{n},\dots, v_K^\ssup{n}$ are in $G_n^*$ asymptotically independent and distributed as $D^\star-1$  
 and we can apply Proposition~\ref{prop:malthusian_dist} and obtain that
\begin{equation}\label{eq:compl_dist}
Z_n:=(d_{G_n^*}(u_n,v_k^{\ssup n})- \lambda_n^{-1}\log n)_{k=1,\dots,K} \Rightarrow    \lambda^{-1} (\log  \cE_k/W_k+ \log 1/\hat W+c)_{i=1,\dots,k}, 
\end{equation}
where $\cE_1,\dots,\cE_K, W_1,\dots,W_K,\hat W$ are independent random variables with $\cE_1,\dots,\cE_K\sim\mathrm{Exp}(1)$ and $W_1,\dots,W_K,\hat W\sim M^*$  and $c$ is as in the proposition. Note that the left hand side of~(\ref{eq:compl_dist}) depends only on $G_n^*$ and $u_n$. Further 
under each~$\P_n^\mathfrak {t}$ the vector $(M(v_1,),\dots,M(v_k))$ is identically distributed to $(W_1,\dots,W_k)$ and thus  applying a Skorokhod coupling we can couple $(G_n^*,u_n)$ and $T$ such that for the coupled random variables, say under $\P^\mathfrak {t}$, 
$$
Z_n\to \lambda ^{-1}(\log  \cE_k/M(v_k)+ \log 1/\hat W+c)_{i=1,\dots,k}=:Z, \ \text{ $\P^\mathfrak {t}$-a.s.}
$$ 
with independent random variables $\cE_1,\dots,\cE_K, \hat W$ (also of $T$) satisfying $\cE_1,\dots,\cE_K\sim\mathrm{Exp}(1)$ and $\hat W\sim M^*$.

We continue arguing with high probability, under the measure $\P^\mathfrak{t}$.
 With high probability,  $u_n$ is not in $V_n|_{<R}$ and
$$
d_{G_n} (o_n,u_n)= \min_{k=1,\dots,K} d_{G_n|_{\leq R}} (o_n,v_k^\ssup{n})+ d_{G_n^*} (v_k^\ssup{n},u_n).
$$
If the latter distance is finite then the corresponding geodesic passes at \hopcount distance $R$ the (a.s.) unique vertex $v_k^{\ssup n}$ for which $d_{G_n|_{\leq R}} (o_n,v_k^\ssup{n})+ d_{G_n^*} (v_k^\ssup{n},u_n)$
 or, equivalently, 
\begin{align}\label{eq93457}
e^{\lambda d_{T}(o,v_k)} \exp\{ \lambda( d_{G_n^*}(v_k^\ssup{n}, u_n)-\lambda _n^{-1} \log n)-c\}
\end{align} is minimal. In that case we denote the respective $k$ by $k_\mathrm{min}^\ssup{n}$ and set $k_\mathrm{min}^\ssup{n}=0$ in the case that  $u_n$ is not connected to $\{v_1^\ssup{n},\dots,v_K^\ssup{n}\}$ in $G_n^*$. 
By construction, (\ref{eq93457})  converges, almost surely, to 
$$
e^{\lambda d_{T}(o,v_k)} \frac {\cE_k}{M(v_k) \hat W}.
$$
We distinguish two cases. If $\hat W$ and at least one of the $M(v_k)$ is strictly positive we denote by $k_\mathrm{min}$ the unique minimizer of the previous term. Otherwise we set $k_{\mathrm{min}}=0$. The distribution of $M(v_k)$ has no atom outside zero so that in the former case the minimizer $k_\mathrm{min}$ is almost surely unique and we have $k^\ssup{n}_\mathrm{min}\to k_\mathrm{min}$, up to nullsets. It remains to show that also in the latter case $k^\ssup{n}_\mathrm{min}\to k_\mathrm{min}$ which follows immediately once we show that
\begin{align}\label{eq94357}
\liminf_{n\to\infty} \P^\mathfrak{t}( k^\ssup{n}_\mathrm{min}=0)\geq  \P^\mathfrak{t}( k_\mathrm{min}=0).
\end{align}
Note that up to nullsets the event $k_\mathrm{min}=0$ agrees with the event 
$$\{\hat W=0\} \cup \{ T(v_1),\dots,T(v_K)\text{ are finite}\}$$ so that the right hand side of~(\ref{eq94357}) satisfies 
$$
\P^\ft(k_\mathrm{min}=0)= (1-\zeta) + \zeta(1-\zeta^*)^{K},
$$
where we recall that $\zeta$ is the survival probability of $T$ and 
$\zeta^*$ the survival probability of a Galton-Watson process with offspring distribution $D^\star -1$.

Conversely, 
\[\begin{aligned}  \p^\ft ( k_{\rm min}^\ssup{n} = 0) & = \p^\ft( u_n \not\leftrightarrow v_i \mbox{ for all } i = 1, \ldots, K ) \\
& = \p^\ft ( v_i \notin C^\ssup{n, *}_{\rm max} \ \forall i = 1, \ldots, K, u_n \in  C^\ssup{n, *}_{\rm max} ) \\
& \qquad + \p^\ft( u_n \not\leftrightarrow v_i \ \forall  i = 1, \ldots, K , u_n \notin C^\ssup{n, *}_{\rm max} ) ,
\end{aligned} \] 
where $C^\ssup{n, *}_{\rm max}$ denotes the largest component in $G_n^*$.
Now the first summand satisfies
\[\p^\ft ( v_i \notin C^\ssup{n, *}_{\rm max} \ \forall i = 1, \ldots, K, u_n \in  C^\ssup{n, *}_{\rm max} ) \ra \zeta ( 1- \zeta^*)^K, 
\]
while the second one satisfies
\[\begin{aligned}  \p^\ft( u_n \not\leftrightarrow & v_i \ \forall  i = 1, \ldots, K , u_n \notin C^\ssup{n, *}_{\rm max} ) \\
 & = \p ( u_n \notin C^\ssup{n, *}_{\rm max} )  - 
\p ( u_n \notin C^\ssup{n, *}_{\rm max} , \exists i =1,\ldots,K\, : \, 
u_n \leftrightarrow v_i \ ) \\
& \ra (1-\zeta) . \end{aligned} \]
Combining these statements yields~\eqref{eq94357}.

Thus we showed that for $k\in\{1,\dots,K\}$
\begin{align*}
\P^\mathfrak {t}(o_n \leftrightarrow v_k^\ssup{n}\text{ is red in }\cG_n) &\to\P^\mathfrak{t}(k_\mathrm{min}=k)\\
&=\P^\mathfrak{t} (\hat W>0)\,  \E^\mathfrak{t}\Bigl[ \frac { e^{-\lambda d_{T}(o,v_k)} M(v_k)}{\sum_{\ell=1}^K  e^{-\lambda d_{T}(o,v_\ell)} M(v_\ell)  } \1_{\{M(v_k) > 0 \}}\Bigr]\\
&= \P\bigl(o \leftrightarrow v_k\text{ is red in }\cT\,|\, T|_{\leq R}=\mathfrak t\bigr),
\end{align*}
 while
\[ \p^\ft(\cG|_{\leq R} \mbox{ is black}\, ) = \p^\ft(k_{\rm min}^n = 0) 
\ra \p^\ft( k_{\rm min} = 0 ) = \p( \cT \mbox{ is black} \, | \, T|_{\leq R} = \ft) . \]
Integrating out the feasible $\mathfrak t$ and recalling that  $\P(\Om_n)\to1$ finishes the proof.
\end{proof}

\section{The case with explosion}\label{sec:explosion}

In this section we will prove Theorem~\ref{thm:local_conv} in the case that
the underlying branching process $T$ explodes in finite time.
The proof relies on a suitable exploration of the graph $G_n$, which 
we will describe in Section~\ref{ssn:exploration}. 
We will only discover local information about the graph, so that
we can carry out the actual proofs for the approximating branching process, which 
will be in Section~\ref{ssn:galton}. Finally, we complete the proof in Section~\ref{ssn:proof_explosion}
using a coupling argument between the graph and the branching process.

\subsection{The exploration process}\label{ssn:exploration}

We use the concept of an exploration process to collect information about a rooted (random or deterministic) weighted graph $(G,o)$. Later we will choose $G=G_n$ or $G=T$. 

The exploration depends on a parameter $R\in\N$. 
We initialise the exploration with the rooted graph $(G,o)|_{\leq R}$ and call all vertices with \hopcount distance $R$ 
 to the root $o$ active. Further we record for each active vertex its degree in $G$. Formally we describe the initial status of the exploration by the tuple $\cE_0=(G^{(0)},\cA_0, \mathfrak d_0)$ with $G\ssup{0}=  (G,o)|_{\leq R}$, $\cA_0=V|_{= R}$ and $\mathfrak d_0=(\mathrm{deg}_G(v))_{v\in \cA_0}$.
The exploration is now defined inductively. To get for $N\in\N$ from $\cE_{N-1}$ to $\cE_{N}$ we find the vertex $v^\star(N)\in \cA_{N-1}$ that minimizes the distance to the root  $o$ and form $G^\ssup{N}$ by adding all immediate neighbours (including the edges with weights) of $v^\star(N)$ in the graph $G$ to $G^\ssup{N-1}$. The new set of active vertices $\cA_{N}$ is formed by discarding $v^\star(N)$ and adding all vertices that were added to get from  $G^{(N-1)}$ to $G^{(N)}$. Further we let $\mathfrak d_N=(\mathrm{deg}_G(v))_{v\in \cA_N}$. We say that the vertex $v^\star(N)$ is \emph{explored}. 

In the case that $\cA_{N-1}$ is empty we set $v^\star(N)=o,\, G^\ssup{N}=G^\ssup{N-1}$ and $\cA_{N}=\cA_{N-1}=\emptyset$.
 For completeness we resolve ties by choosing one of the minimisers uniformly at random.
 Finally, we set $\cE_N = (G^\ssup{N}, \cA_N, \mathfrak d_N)$.

Supposing that all vertices have distinct distances to the root $o$ and that $\cA_{N-1}\not=\emptyset$ (or, equivalently, $v^\star(N)\not= o$), the set $\cA_N$ ($N\in\N$) consists of all vertices $v\in V$ satisfying either
\begin{itemize}
\item[(A1)] $d(o,v-)  \leq d(o,v^\star(N))< d(o,v)$ and  the \hopcount distance from $o$ to $v$ is strictly  greater than $R$, or
\item[(A2)]  $ d(o,v^\star(N))< d(o,v)$ and the \hopcount distance from $o$ to $v$ is $R$.
\end{itemize}

Supposing that geodesics are unique we note that for every vertex $v\in V|_{\geq R}$ the geodesic from $o$ to $v$ passes through a unique  vertex in $V|_{=R}$ and we denote by $v^\star_R(N)$ the  vertex that is traversed by the geodesic to $v^\star(N)$. If  $v^\star(N)= o$ we set $v^\star_R(N)=o$. We call the set of vertices from $V|_{\geq R}$ whose geodesics pass through a vertex $v\in V|_{=R}$ the branch of $v$.

Based on a parameter $\eps>0$ we distinguish four kinds of active vertices:
\begin{enumerate}
\item  $v\in \cA_N$ is in the branch of  $v^\star_R(N)$ and satisfies
$d(o,v)\leq d(o,v^\star(N))+\eps$
\item  $v\in \cA_N$ is in the branch of  $v^\star_R(N)$ and satisfies
$d(o,v)> d(o,v^\star(N))+\eps$
\item  $v\in \cA_N$ is not in the branch of  $v^\star_R(N)$ and satisfy
$d(o,v)\leq d(o,v^\star(N))+\eps$
\item  $v\in \cA_N$ is not in the branch of  $v^\star_R(N)$ and satisfy
$d(o,v)> d(o,v^\star(N))+\eps$
\end{enumerate}
We will later see that in the explosive regime an exploration of large configuration models $G_n$ yields for $\eps>0$ small and $N$ large configurations that  typically feature a large number of vertices of type (i), no vertices of type (iii) and a relatively small number of vertices of type (iv). See Figure~\ref{fig:idea} for an illustration of the exploration process.

\begin{figure}[ht!]
\begin{center}
\includegraphics[width=0.9\textwidth]{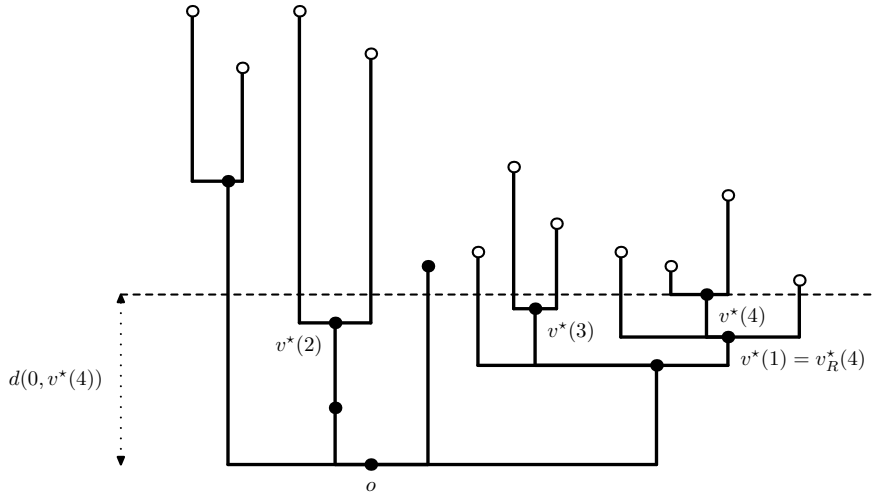}% 
\caption{The exploration process with $R =2$ after $N = 4$ iterations of the exploration process for a realisation of $T$. The vertical distances between the vertices correspond to the edge weights. The black vertices have been explored, while the white vertices are in $\cA_4$.}
\label{fig:idea}
\end{center}
\end{figure}

Our analysis relies on two  facts about the exploration of configuration models  with asymptotic degree distribution $\cD$ (as defined above):
\begin{enumerate}
\item For every $N\in\N$ there exists a coupling of the neighbourhoods $(G_n,o_n)$ of the configuration model 
with the random rooted tree $T$ such that with high probability the corresponding $N$-step explorations $\cE_N^{(n)}$ and $\cE_N$ of $(G_n,o_n)$ and $T$ are isomorphic. More explicitly, there exists a coupling together with events $\Om_n\subset \Om$ and a random  mapping $\varphi_n$ such that $\P(\Om_n)\to1$ and  for every $n\in\N$ and $\om\in\Om_n$
\begin{itemize} \item $\varphi_n(\om,\cdot)$ is a bijection between the sets of vertices of $T^{(N)}$ and the ones  of $G_n^{(N)}$ and
\item the $N$-step exploration  $\cE_N^{(n)}$ is obtained from $\cE_N$ by relabeling the vertices by application of $\varphi_n(\om,\cdot)$.
\end{itemize}
The coupling can be achieved in a Markovian way meaning that the conditional law of $(G_n,o_n)$ given $\cE_N^{(n)}$ is identical with the one given $(\cE_N^{(n)},\cE_N)$. This means that the coupling does not reveal additional information about the undisclosed parts.
\item Define $G_n^\star(N)$ as the weighted graph obtained from $G_n$ by removing all edges appearing in $\cE_N^{(n)}$.
Then, given the exploration $\cE_N^{(n)}$ (and hence also given $(\cE_N^{(n)},\cE_N)$), the graph $G_n^\star(N)$ 
is a configuration model with random degree sequence.  Note that all vertices that have been explored in the first~$N$ steps   have degree $0$ in $G_n^\star(N)$. Further in the case that $G_n^\ssup{N}$ is a tree, every active vertex $v$ has degree $\mathrm {deg}_{G_n}(v)-1$ in $G_n^\star(N)$.
\end{enumerate}

\subsection{Auxiliary lemmas for the branching process}\label{ssn:galton}

First we prove some auxiliary results for branching processes with explosion.
Suppose that $T$ as introduced in~\ref{ssn:main_results} has with strictly positive probability a finite explosion time
$$
\tau= \inf\{ t \geq 0 \, : \, \# \{ v : d_T(o,v) < t \} = \infty \}.
$$
For $v\in V$, we write $|v|$ for the \hopcount distance between $v$ and the root $o$. Further, we denote for $v\in V\backslash \{o\}$  by $v-$ the unique neighbouring vertex in $\gw$ with $|v\!-|=|v|-1$.  
 We call the distance 
\[ S_v := d_{T}(o, v), \]
the \emph{birth time} of the vertex $v$ and denote by $\cS$ the  event that 
\[ \cS = \{ \tau < \infty \} , \]
which as we will see is indistinguishable from the event that $T$ survives.

Our analysis is based on the  following properties for explosive Galton-Watson processes, which follow easily from 
the fact that the distribution $\mu$ has no atoms.
For a proof of the last claim, see~\cite[Claim~2.4]{komjathy_explosion}.

\begin{lemma}\label{le:prop_explosion}
Almost surely, all times
$(S_v:v\in V)$ are distinct and the distribution of $\tau$ has no atoms beside $\infty$. Further on $\{\tau<\infty\}$ there is exactly one infinite random  ray $\xi = (o = \xi_0, \xi_1, \ldots )\in V^\infty$,
with $\xi_{n-1}=\xi_n-$ for $n\in\N$ and
$$
\tau =\lim_{\ell\to\infty} S_{\xi_\ell}.
$$
Moreover, conditionally on the event  $\{ \#T = \infty\}$
 we have  $\tau < \infty$, almost surely. 
\end{lemma}

We define
\[ V|_{=R}   = \{ v \in V : \, |v| = R \} \, . \] 
and 
\[ V|_{\geq R} = \{ v \in V : \, |v| \geq R \} , \] 
and $V|_{> R}= V|_{\geq R} \setminus V|_{=R}$. 
Finally, for every $v \in V$ let $T(v)$ denote the subgraph obtained from $T$ by keeping the vertex $v$ and all its descendants and by declaring $v$ as root.  We denote by $V(v)$ the respective set of vertices.

We apply the exploration introduced in Section~\ref{ssn:exploration} for the rooted random graph $(G,o)=T=(V,E,o)$. In particular, we use the notation for $v^\star(N)$ and $v^\star_R(N)$ as introduced there. 
In a first step we show that for $\eps>0$ and sufficiently large $N$, typically, the number of edges in the branch of $v_R^\star(N)$ that leave $G^{(N)}$ is large and further that the explosion occurs in the branch of $v_R^\star(N)$.

\begin{lemma}\label{le:gw1}\label{cor:number_good_guys}   Let  $R\in\N$ and  $\eps>0$ and consider for $N\in\N$
$$
A_{R,\eps}(N)=\{  v \in V (v_R^\star(N)) \, : \,  S_{v-} < S_{v^\star(N)}\leq S_v<S_{v^\star(N)}+\eps \}.
$$
Then one has 
$$
\lim_{N\to\infty} \p(\mbox{explosion traverses $v_R^\star(N)$}\,|\,  v^\star(N) \neq o)=1
$$
and
for every $\ell, R \in \N, \eps>0$, 
$$\lim_{N\to\infty} \p \Big( \sum_{v \in A_{R,\eps}(N)} (\deg_\gw(v) -1)  \leq \ell \, \Big| \, v^\star(N) \not =o	 \Big) =0.$$
\end{lemma}

\begin{proof} 
Recall that $\xi$ is the ray leading to explosion. Then, we have that
\[ \p (v^\star(N) \neq o, v^\star_R(N) \neq \xi_R) 
\leq \P ( \tau < \infty, v^\star(N) \notin \xi) + \p ( v^\star(N) \neq 0, \tau = \infty) . \]
As $N \ra\infty$ the first probability on the righ hand side tends to zero by the definition of $v^\star(N)$, while the second 
probability tends to zero since $\bigcap_{N \in \N} \{ v^\star(N) \neq o \} = \{ \# V = \infty \}$ 
and by Lemma~\ref{le:prop_explosion} the latter event is indistinguishable from $\cS = \{ \tau  < \infty\}$. 
In particular, since $\p(\cS) > 0$ this shows the first part of the lemma. 

For the second part, 
we first show that for every $\ell>0$  \begin{align}\label{eq83576} \lim_{N\to\infty} \p ( \#A_{R,\eps}(N) \leq \ell \, | \, v^\star(N) \neq o) =0.\end{align}
By the same argument as in the first part, it suffices to prove the statement conditionally on $\{ \# V = \infty \}$ rather than $v^\star(N) \neq o$.

Consider the following events
$$ 
E_1(N):=\{ v^\star(N)\not= o, \text{no explosion among offspring of }  v^{\star}_R(N) \mbox{ before time } S_{v^\star(N)}+\eps\}
$$
and
$$
E_2(N):=\{v^\star(N)\not= o, \text{explosion ray does not traverse }v^\star_R(N)\}.
$$
Note that $\lim_{N\to\infty} \P(E_2(N))= 0$ by the first part. 
Furthermore, note that also $\P(E_1(N))$ tends to $0$ since
$$
\P(E_1(N)) \leq \P(E_2(N)) + \P( \tau>S_{v^\star(N)}+\eps)\to 0.
$$
For $v\in V$ we let  $\tau(v)$ denote the explosion time of the subtree $T(v)$. Then
$$
 E(N):=\{v^\star(N)\not=o \text{ and } \forall v\in A_{R,\eps}(N): \tau(v)>\eps\} \subset E_1(N)\cup E_2(N)
$$
and given  $A_{R,\eps}(N)$ and $\{v^\star(N)\not= o\}$ the explosion times $(\tau(v): v\in A_{R,\eps}(N))$ form a sequence of independent and identically distributed random variables
with the same distribution as the random variable $\tau^*$, which is the explosion time if the underlying Galton-Watson tree has
offspring distribution $D^\star-1$ throughout. We conclude that
\begin{align*}
 \E[ \1\{v^\star(N)\not=o\} \, \P(\tau^*>\eps)^{\# A_{R,\eps}(N)}]=\P(E(N)) \leq \P(E_1(N))+\P(E_2(N))\to 0.
\end{align*}
Using that $\{ \# V = \infty \} \subset\{v^\star(N)\not=o\}$ we get that 
$$
\1_{\{ \# V  =\infty \}} \, \P(\tau^*>\eps)^{\# A_{R,\eps}(N)}\to 0, \text{ in probability},
$$
which implies that $\# A_{R,\eps}(N)\to \infty$, in probability, on $\{ \# V = \infty \}$.

The second statement follows immediately from property~(\ref{eq83576}) by noting that given $A_{R,\eps}(N)$ and $\{v^\star(N)\not=o\}$ the degrees $(\deg_T(v):v\in A_{R,\eps}(N))$ form a vector of independent random variables with the same distribution as $D^\star-1$.
\end{proof}

 As we have seen above in the event of explosion one typically has explosion along the vertex $v^\star_R(N)$ and the descendants   in $A_{R,\eps}(N)$ have a diverging number of stubs  that leave $G^{(N)}$.
Let us now consider the active vertices belonging to other branches. 
For $N\in\N$ we let 
\[  A'_R(N) = \cA_N \setminus V(v_R^\star(N))  . \]

\begin{lemma}\label{le:eps}
 For every $R \in \N$, 
\[ \lim_{\eps \downarrow 0 } \limsup_{N \ra \infty} \p (v^\star(N)\not= o,  \exists v \in A'_R(N) \, : \,  S_v < S_{v^*(N)} + \eps ) = 0   \]
and 
\[ \lim_{ \kappa \ra \infty} \limsup_{N \ra \infty} \p \Big(v^\star(N)\not=o,  \sum_{v \in A'_R(N) } (\deg_{\gw}(v)-1) \geq \kappa \Big) = 0 . \]
\end{lemma}

\begin{proof} 
For a vertex $v\in V$ we denote by
$$
\tau(v)=\inf\{t\geq 0: \#\{w\in V(v): S_w\leq t\}=\infty\}
$$
the explosion time along the vertex $v$ (which may be infinite). Since given   $T|_{\leq { R}}$ the subtrees $(T(v):v\in V|_{={R}})$  are independent with all birth and explosion times being absolutely continuous (besides an atom in $\infty$), almost surely, 
all birth times and explosion times not equal to $\infty$ are pairwise distinct. We continue arguing conditionally on $\{\tau<\infty\}$. In this case there exists a random  $\eps_1>0$ such that for the vertex $\xi_R\in V|_{=R}$ on the ray to explosion one has
$$
\tau(\xi_R)+\eps_1<\min_{w\in V[=R]\backslash \{\xi_R\}} \tau(w).
$$
Consequently, there are at most finitely many vertices in $\bigcup_{w\in  V|_{\geq R}\setminus V(\xi_R) 
}  V(w)$ with birth time in $(0,\tau(\xi_R)+\eps_1)$ and since none of these equals $\tau=\tau(\xi_R)$ there exists a random $\eps_0>0$ such that
all birth times of the vertices in $\bigcup_{w\in  V|_{ \geq R}\setminus V(\xi_R)%
} V(w)$ do not intersect $(\tau-\eps_0,\tau+\eps_0)$.
Note that occurrence of the  event $\{ \exists v\in A'_R(N) \, : \,  S_v < S_{v^*(N)} + \eps \}\cap \cS$
implies that $\tau<\infty$ and that at least one of the following events holds:
\begin{itemize}
\item $v^\star_R(N)\not= \xi_R$,
\item $S_{v^\star(N)}<\tau-\eps$,
\item $\eps_0<\eps$.
\end{itemize}
The probabilities of the first and second event tend to zero as $N\to\infty$: the first one due to  Lemma~\ref{le:gw1} and the second one due to pointwise convergence  $S_{v^\star(N)}\to\tau$  on $\{\tau<\infty\}$. Consequently,
\begin{equation}\label{eq:0610-2}
\limsup_{N \ra \infty} \p (\cS\cap\{  \exists v \in A'_R(N) \, : \,  S_v < S_{v^*(N)} + \eps  \}) \leq  \P(\cS \cap \{\eps_0<\eps\})
\end{equation}
and the first statement follows by letting $\eps\downarrow 0$.

Note that in the case when $\cS\cap \{v^\star_R(N)=\xi_R\}$ holds we have 
$$
 \sum_{v \in A'_R(N) } (\deg_{\gw}(v)-1) \leq \!\!\!\sum_{v\in V|_{>R}\backslash V(\xi_R)\,:\,S_{v-}<\tau}(\deg_{\gw}(v)-1) +  \sum_{|v|=R, v \neq \xi_R} (\deg_{\gw}(v) -1).
$$
The term on the right hand side is almost surely finite and does not depend on $N$. Hence the second statement follows by observing that $\p(\{v^\star(N)\not =o\}\backslash \cS)$ and $\p(\cS, v^\star_R(N)\not =\xi_R)$ tend to zero as $N\to\infty$.
\end{proof}

\subsection{Proof of Theorem~\ref{thm:local_conv} in the case of explosion}\label{ssn:proof_explosion}

We now complete the proof of Theorem~\ref{thm:local_conv}, where we use the notation 
of the previous sections.  In particular, $(\cE_N)_{N \in \N_0}$ denotes the exploration of 
$T$, whereas $(\cE_N^\ssup{n})_{N \in \N_0}$  refers to the exploration of $(G_n,o_n)$. Moreover, the $N$-th explored vertex in $\cE_N$ is denoted by $v^\star(N)$ and $v^\star_n(N)$  refers to the corresponding vertex in the exploration $\cE_N^\ssup{n}$.

\begin{proof}[Proof of Theorem~\ref{thm:local_conv}  in the case of explosion.]
We generate $\cT$ by first generating $T$ and an independent Bernoulli random variable $I$ with success probability $\P(\cS)$. In the case that $\cS$ occurs and $I=1$ we mark the unique explosion ray from $o$ to infinity in red. Otherwise all edges of the rooted tree $T$ are coloured  black.

Fix $R \in \N$  and $\delta \in (0,1)$ throughout. 
We apply the exploration $(\cE_N:N\in\N_0)$ as introduced in Section~\ref{ssn:exploration} for the random rooted tree $T$. We consider the events 
 \begin{itemize}
 \item $E_{-1}=E_{-1}^\ssup{N}=\{\1_{\{v^\star(N)\neq o\}}=\1_\cS\}$
 \item $E_0=E_0^\ssup{N}=\{v^\star(N)=o\text{ or  explosion ray traverses }v_R^\star(N)\}$
\item $E_1=E_1^\ssup{N}=  \{ v^\star(N)= o\text{ or } \forall v \in A'_R(N) \, : \, S_v \geq S_{v^\star(N)} + \eps \}$
 \item $E_2=E_2^\ssup{N}=\{ \sum_{v \in A'_R(N)} (\deg_\gw(v)-1) \leq \kappa_2 \}$
 \item $E_3=E_3^\ssup{N}=\{v^\star(N)= o\text{ or }  \sum_{v \in A_{R,\eps}(N)} (\deg_\gw(v) -1) \geq \kappa_3 \}$
 \end{itemize}
 for parameters $\kappa_0,\kappa_2,\kappa_3,\eps>0$ to be fixed within the next lines. 
 
 Since $\bigcap_{N\in\N} \{v^\star(N)\neq o\}=\{\#V=\infty\}$ and the latter event is indistinguishable from $\cS$ we have $\p(E_{-1}^{\ssup{N}})\geq 1-\delta$ for all sufficiently large $N\in\N$. 
By Lemma~\ref{le:gw1} the probability of $E_0^\ssup{N}$ tends to one and thus $\p(E_0^\ssup{N})\geq 1-\delta$ for all sufficiently large $N$. 
By Lemma~\ref{le:eps} we can fix  $\eps>0$ and $\kappa_2\in\N$ such that $\p(E_1^\ssup{N})\geq 1-\delta$ and $\p(E_2^\ssup{N})\geq 1-\delta$ for all sufficiently large $N$.  We choose $\kappa_3=\frac {1-\delta}{\delta}\kappa_2$ and note that  by Lemma~\ref{cor:number_good_guys} $\p(E_3^\ssup{N})\geq 1-\delta$  for sufficiently large $N\in\N$. 
 Further by standard arguments one has for sufficiently large $N$ for sufficiently large $n\in\N$ that the event 
 \begin{itemize} 
 \item $E_4=E_4^\ssup{N,n}=\{\text{either }v_{n}^\star(N)= o_n \text{ or }o_n\in C^\ssup{n}_\mathrm{max}\}$
\end{itemize}
satisfies $\p(E_4^\ssup{N,n})\geq 1-\delta$.
 We now fix $N\in\N$  such that all the above events occur at least with probability $1-\delta$ for sufficiently large $n\in\N$.

Note that we can couple  $(T,I)$ with each individual random rooted graph $(G_n,o_n)$ in a Markovian manner such that for the respective explorations  $\lim_{n\to\infty} \p(\cE_N^\ssup{n}\sim \cE_N)=1$. Consequently, one has for sufficiently large  $n\in\N$ that
\begin{itemize}
\item $E_5=E_5^\ssup{N,n}=\{\cE_N^\ssup{n}\sim \cE_N\}$
\end{itemize}
satisfies $\p(E_5^\ssup{N,n})\geq 1-\delta$.  Given $(T,G_n,o_n,I)$ we generate a random vertex $u_n'$ of $V_n$ as follows: 
if $I=1$ we pick $u_n'$  uniformly at random from $C^\ssup{n}_\mathrm{max}\backslash \cE_N^\ssup{n}$ and if  $I=0$
 uniformly at random from $V_n\backslash (C^\ssup{n}(o_n)\cup C^\ssup{n}_\mathrm{max})$,  where $C^\ssup{n}(o_n)$ denotes the connected component of $G_n$ containing $o_n$. In the case that one of the latter sets is empty we assign $u_n'$ a dummy value. 
 We note that the total variation distance between $(G_n,o_n,u'_n)$ and $(G_n,o_n,u_n)$ (with $u_n$ a uniform vertex of $V_n$, independently chosen of $(G_n,o_n)$)  is bounded by
$$
\frac 1n \E\bigl[\#\{\text{vertices in $\cE_N^n$}\}+\ind_{\{o_n\not\in C_\mathrm{max}^\ssup{n}\}} \# C^\ssup{n}(o_n)+ |\#C^{(n)}_\mathrm{max}-n \P(\cS)|\bigr] 
$$
which tends to zero as $n\to\infty$. On a sufficiently rich probability space we can alter the definition of  $u_n$ such that the distribution of $(G_n,o_n,u_n)$ remains unchanged and 
$u_n=u_n'$,  with high probability.
We denote
\begin{itemize}
\item $E_6=E_6^\ssup{N,n}=\{u_n=u_n'\}$.
\end{itemize}
We are now in the position  to fix $n\in\N$ sufficiently large such that the probability of the events $E_4,E_5,E_6$ exceed $1-\delta$.

We introduce one additional event
\begin{itemize}
\item $E_7=E_7^\ssup{n}=\{o_n\not \in C_\mathrm{max}^\ssup{n}\text{ or }  I=0 \text{ or }$\\
\phantom{1235}\qquad $o_n \text{ and } u'_n \text{ are connected by a geodesic passing through }v_{R,n}^\star(N) \}$.
\end{itemize}

First we show that in the case that all events $E_{-1},E_0, E_4, E_5, E_6$ and $E_7$ occur the corresponding random isomorphism $\varphi_n$ taking $\cE_N$ to $\cE_N^\ssup{n}$ also maps the coloured tree $\cT|_{\leq R}$ to the coloured neighbourhood $(G_n,o_n,c_n)|_{\leq R}$. 
Indeed, it takes $T|_{\leq R}$ to $(G_n,o_n)|_{\leq R}$ and to verify that also the colourings  coincide we distinguish three different cases. If $\cS$ occurs  then $o_n$ is in $C_{\mathrm{max}}^\ssup{n}$ ($E_{-1}$, $E_4$). If additionally $I=1$, then in $\cT|_{\leq R}$ we have coloured the geodesic from $o$ to $v_R^\star(N)$ in red ($E_{-1}$, $E_0$). Moreover, $(G_n,o_n,c_n)|_{\leq R}$ has a coloured geodesic from $o_n$ to $u_n=u_n'$ passing through $v_{R,n}^\star(N)=\varphi_n(v_{R}^\star(N))= \varphi_n(\xi_R)$ ($E_0$, $E_6$, $E_7$).
Conversely if $\cS$ and $I=0$ occur the graph $\cT$ is black by definition and $(G_n,o_n,u_n)$ is black since $u_n=u_n'\not \in C^\ssup{n}(o_n)$ ($E_6$) and definition of $u_n'$. 
It remains to consider the case that $\cS^c$ occurs. In that case all edges in $\cT$ are black. Further since $v^\star(N)=0$ ($E_{-1}$) we have $v_n^\star (N)=o_n$ ($E_5$) and hence     $o_n\not\in C_\mathrm{max}^\ssup{n}$ ($E_4$). By definition of $u_n'$ we have $u_n=u_n'\not \in C^\ssup{n}(o_n)$ ($E_6$) so that also all edges of  $(G_n,o_n,u_n)$ are coloured in black.

It remains to show that for every $\eta>0$ we can choose $\delta>0$ sufficiently small to guarantee that
$$
\P(E_{-1}\cap E_0 \cap  E_4 \cap  E_5 \cap  E_6 \cap E_7)\geq 1-\eta.
$$
By definition of the events one has
$$
\P\Big(\bigcap_{i=-1}^6 E_i\Big)\geq 1- 8\delta
$$
and it remains to control the probability of  $E_7$. Note that   $E_1\cap E_2\cap E_3\cap E_5$ is in the $\sigma$-field  $\cF_N^\ssup{n}=\sigma(\cE_N,\cE^\ssup{n}_N)$. Further conditionally on the latter $\sigma$-field the graph that one obtains by removing all edges appearing in $\cE^\ssup {n}_N$ from $G_n$ yields a configuration model $G_n^\star(N)$ with random degree sequence. Note that under the conditional distribution the degrees of the \emph{active} vertices are deterministic. 

We continue arguing under the conditional distribution for a fixed realisation of the exploration for which $E_1,E_2,E_3$ and $E_5$ occur, say under the measure $P$. 
In analogy to $A_{R,\eps}$ and $A_{R}'$ we denote by $A_{R,\eps,n}$ and $A_{R,n}'$ the respective active vertices of the configuration $\cE_N^\ssup{n}$. 
We denote by $H$ the unique  half-edge in $G_n^\star(N)$ attached to a vertex in $A_{R,\eps,n}\cup A'_{R,n}$ traversed by the geodesic connecting $u'_n$ with $A_{R,\eps,n}\cup A'_{R,n}$ in $G_n^\star(N)$ provided such a geodesic exists. If such a geodesic does not exist we set $H=\partial$ for a dummy variable $\partial$. 

We denote by $\cH_{R,\eps,n}(N)$ and $\cH_{R,n}'(N)$  the set of all half-edges attached to the vertices  $A_{R,\eps,n}$ and $A'_{R,n}$ in $G_n^\star(N)$, respectively, and set
$\cH_n=\cH_{R,\eps,n}(N)\cup \cH_{R,n}'(N)$. For two distinct half-edges $h_1,h_2 \in \cH_n$ we associate the graph $G_n^\star(N)$ with a graph $G_n^{h_1,h_2}(N)$ that is obtained as follows: suppose that $h_1$ and $h_2$ form  in $G_n^\star(N)$ an edge with the half-edges $g_1$ and $g_2$, then we obtain $G_n^{h_1,h_2}(N)$ by rewiring the edges $\langle g_1,h_1\rangle$ and $\langle g_2,h_2\rangle$ as $\langle g_1,h_2\rangle$ and $\langle g_2,h_1\rangle$ in $G_n^\star(N)$, respectively. In the process the weights are kept. In particular, in the case where $h_1$ and $h_2$ form an edge, nothing changes. Since a configuration model is obtained by a uniform pairing of all half-edges, the latter operation has no effect on the distribution so that under $P$, $G_n^\star(N)$ and $G_n^{h_1,h_2}(N)$  are identically distributed. 
Note that the event $\{o_n\in C_\mathrm{max}^{\ssup n}, I=0\}$ occurs for $G_n^\star(N)$ if and only if it occurs for the rewired graph $G_n^{h_1,h_2}(N)$. Further one has 
$H=h_1$ for $G_n^\star(N)$ if and only if $H=h_2$ for the rewired graph  $G_n^{h_1,h_2}(N)$. Consequently, the value of
\begin{align*}
P(o_n \in C^{\ssup n}_\mathrm{max}, I=1, H=h)
\end{align*}
is constant for all $h\in\cH_n$ so that given $\tilde E_7=\tilde E_7^\ssup{n}:=\{o_n \in C^{\ssup n}_\mathrm{max}, I=1\}$  and $\{H\not=\partial\}$, $H$ is uniformly distributed on $\cH_n$. 
Note that by $E_1$ in the case that $\tilde E_7=\tilde E_7^\ssup{n}$ and $\{H\in A_{R,\eps,n}\cup\{\partial\}\}$ occur there exists a geodesic from $u_n'$  to $o_n$ and it traverses $v_{R,n}^\star(N)$. Consequently, by $E_2$, $E_3$ and choice of $\kappa_3$
$$
P(E_7) \geq \frac{\# \cH_{R,\eps,n}(N)}{\#\cH_n}\geq \frac {\kappa_3}{\kappa_2+\kappa_3}=1-\delta.
$$
Altogether we thus get 
$$
\P(E_1\cap E_2\cap E_3\cap E_5\cap E_7)\geq 1- 5\delta 
$$
and, finally,
$$
\P\Big(\bigcap_{i=-1}^7 E_i\Big)\geq 1- 9\delta, 
$$
which completes the proof of the theorem.
\end{proof}

%%%%%%%%%%%%%%%%%%%%%%%%%%%%%%%%%%%%%%%%%%%%%%%%%%%%%%%%%%%%%%%%%%%%%%%%%%%%%%
%	Bibliography 																	 %
%%%%%%%%%%%%%%%%%%%%%%%%%%%%%%%%%%%%%%%%%%%%%%%%%%%%%%%%%%%%%%%%%%%%%%%%%%%%%%
% 
%\setlength{\bibsep}{0.1\baselineskip}
\bibliographystyle{plain}
%\bibliography{random_graphs}

\end{document}